\documentclass{amsart}
\usepackage{amssymb,amsmath,amsfonts,amsthm}
\usepackage{setspace}
\usepackage{amscd,enumerate,amsfonts,calc,amsmath,verbatim,xypic}

\newcommand{\Ext}{\operatorname{Ext}}
\newcommand{\Tor}{\operatorname{Tor}}
\newcommand{\Soc}{\operatorname{Soc}}
\newcommand{\Hom}{\operatorname{Hom}}

\newcommand{\xra}{\xrightarrow}
\newcommand{\lra}{\longrightarrow}

\newcommand{\fm}{{\mathfrak m}}

\theoremstyle{remark}

\swapnumbers

\theoremstyle{plain}
\newtheorem{theorem}{Theorem}[section]

\newtheorem{prop}[theorem]{Proposition}
\newtheorem{lemma}[theorem]{Lemma}
\newtheorem{cor}[theorem]{Corollary}
\newtheorem{defn}[theorem]{Definition}

\theoremstyle{definition}
\newtheorem{chunk}[theorem]{}

\newtheorem*{chunk*}{}
\newtheorem{exam}[theorem]{Example}

\theoremstyle{remark}

\newtheorem{remark}[theorem]{Remark}

\newtheorem*{nota}{Notation}

\numberwithin{equation}{theorem}

\numberwithin{subchunk}{theorem}

\begin{document} \title[Vanishing of Tor]{Betti Numbers for Modules Over Artinian Local Rings}
\author[K. He]{Kaiyue He}
\address{Mathematics Department, Syracuse University, NY 13244}
\email{khe110@syr.edu}

\date{\today} 
\keywords{Vanishing of (co)homology, freeness criteria, Auslander-Reiten conjecture, Betti Numbers.}
\subjclass[2020]{Primary 13C12, 13C14, 13D22, 13D30, 13H99; Secondary 13H10, 13D02}
\maketitle

\begin{abstract}
We introduce a new numerical invariant $\gamma_I(M)$ associated to a finite-length $R$-module $M$ and an ideal $I$ in an Artinian local ring $R$. This invariant measures the ratio between $\lambda(IM)$ and $\lambda(M/IM)$. We establish fundamental relationships between this invariant and the Betti numbers of the module under the assumption of the $\Tor$ modules vanishing. In particular, we use this invariant to establish a freeness criterion for modules under certain $\Tor$ vanishing conditions. The criterion applies specifically to the class of $I$-free modules --- those modules $M$ for which $M/IM$ is isomorphic to a direct sum of copies of $R/I$. Lastly, we apply these results to the canonical module, proving that, under certain conditions on the ring structure, when the zeroth Betti number is greater than or equal to the first Betti number of the canonical module, then the ring is Gorenstein. This partially answers a question posed by Jorgensen and Leuschke concerning the relationship between Betti numbers of the canonical module and Gorenstein properties.
\end{abstract}

\section*{Introduction}

The study of Betti numbers plays a fundamental role in understanding the structure of modules over Noetherian rings. In particular, they provide valuable insights into the behavior of minimal free resolutions. Minimal resolutions over local rings have been extensively studied, and a central problem in this area is understanding the asymptotic growth of Betti numbers. Avramov \cite{avramov1998infinite} summarized key results concerning measures for the asymptotic size of resolutions, originally introduced by Alperin and Evens \cite{alperin1981representations} including the notions of complexity and curvature of a module. While the behavior of these invariants is well understood for Golod rings and complete intersections, less is known for general Cohen-Macaulay rings. \\

A key approach, inspired by various works on the Auslander-Reiten conjecture, is to examine pairs of modules $(M, N)$ for which $\Ext^i(M, N) = 0$ or $\Tor_i(M, N) = 0$ for some values of $i$. Several results suggest that, in such cases, the Betti numbers of $M$ and $N$ are closely related. The vanishing of these $\Tor$ and $\Ext$ modules is particularly significant, as it can have implications for the structure of the ring or the freeness of one of the modules. For example, Jorgensen and Leuschke \cite{jorgensen2007growth} showed that if $\Tor_i(M, N)$ vanishes for certain values of $i$, then the Betti growth of $N$ is bounded above by specific invariants of $M$. Furthermore, they demonstrated that under certain conditions, the vanishing of $\Tor$ with one of the modules being the canonical module implies that the ring is Gorenstein. Huneke, Sega, and Vraciu \cite{huneke2014vanishing} used a similar approach to establish connections between the vanishing of $\Tor$ and the growth of Betti numbers. They also proved that for local rings $(R, \fm)$ with $\fm^3 = 0$, the vanishing of certain $\Tor$ modules implies the freeness of one of the modules. \\

In this paper, we extend the framework developed in \cite{huneke2014vanishing} by introducing and analyzing a new invariant, $\gamma_I(M)$ (see Definition \ref{definition}), associated with an ideal $I$ of an Artinian local ring $R$ and a finite $R$-module $M$. Our results generalize to maximal Cohen-Macaulay modules over Cohen-Macaulay local rings by using the standard depth reduction techniques. This invariant provides a refined measure of the interplay between the module length and Betti growth. \\

In Section \ref{Betti numbers}, we establish fundamental results on Betti numbers and explore key properties of the invariant $\gamma_I(M)$. In particular, we prove that under the vanishing of two consecutive $\Tor$ modules and a quotient ring freeness condition (which we refer to as $I$-freeness; see Notation \ref{IFree}), the ratio of consecutive Betti numbers of a module remains fixed (Lemma \ref{bettiandgamma}), forming the foundation for the rest of the paper. \\

In Section \ref{Vanishing of Tor and Freeness of A Module}, we investigate how the vanishing of certain $\Tor$ modules affects the freeness of a module via our new invariant. We show that if three consecutive $\Tor$ modules vanish under appropriate $I$-freeness conditions, then the invariant $\gamma_I(M)$ satisfies a quadratic equation (Theorem \ref{ThreeVanishP}). Furthermore, we generalize the results of \cite{huneke2014vanishing} to local rings $(R, \fm)$ with an ideal $I$ satisfying $\fm^2 I = 0$, or more generally to modules annihilated by $\fm I$. \\

In Section \ref{Betti Numbers of the Canonical Module}, we apply our results to the canonical module $\omega$, providing conditions under which the invariant characterizes Gorenstein rings. In particular, we address a question posed in \cite{jorgensen2007growth}: given a ring $R$ with canonical module $\omega$, does $b_0(\omega) \leq b_1(\omega)$ imply that $R$ is Gorenstein? We obtain a partial answer to this question (Corollaries \ref{cor3.4} and \ref{cor3.5}). \\

Through these results, our work contributes to a deeper understanding of Betti number growth and its relationship to module structure, particularly in the context of vanishing $\Tor$ conditions. \\

\section{Betti numbers}
\label{Betti numbers}

In this paper $(R,\fm)$ denotes a commutative Artinian local ring
with a unique maximal ideal $\fm$. Let $I$ denote an ideal of $R$. For finite modules, we mean finitely generated modules.
The number $b_0(M)$ denotes the minimal number of generators of a finite module $M$ and
$\lambda(M)$ denotes the length of the module $M$.

\begin{chunk}
    \label{minimalres}

For every $i\ge 0$ we let $M_i$ denote the $i$-th syzygy of $M$ in a
minimal free resolution \[ \dots\lra R^{b_{i+1}(M)}\xrightarrow{\
\delta_i\ } R^{b_{i}(M)}\lra \dots\xra{\ \delta_0\ } R^{b_0(M)}\lra
M\lra 0.\]

In particular, the resolution satisfies $M_i \subset \fm R^{b_{i-1}(M)}$ for all $i \geq 1$.
\end{chunk}
The number $b_i(M)$ is called the $i$-th {\em Betti number} of $M$
over $R$. 

\begin{defn} \label{definition} For each nonzero $R$-module $M$ of finite length we set
\[
\gamma_I(M)=\frac{\lambda(M)}{\lambda(M/IM)}-1.
\]
\end{defn}

Huneke, \c Sega and Vraciu studied the case of $I = \fm$ in \cite{huneke2014vanishing}.
While studying the results in \cite{huneke2014vanishing}, we find that the critical property involved in the proof of all properties involving $\gamma_{\fm}(M)$ is the fact that for any module $M$, $M/\fm M$ is always free over $R/\fm$. Hence, their results can generalize to modules $M$ such that $M/IM$ is free over $R/I$ for a fixed ideal $I$.

\begin{chunk} For simplicity, we give the following notation:
\label{IFree}
\begin{nota}
    For an ideal $I$ and a module $M$, we say that $M$ is \textit{$I$-free} if $M/IM$ is free over $R/I$ as a $R/I$-module. 
\end{nota}
    In practice, it is not hard to check that a module $M$ is $I$-free.
    One can simply look at the presentation matrix $T_M$, of a module $M$. The module $M$ will always be \textit{$I$ free} for the ideals $I$ that contain the ideal generated by the entries of $T_M$.  

\end{chunk}

We will often use the definition of $\gamma_I(M)$ in length
computations as follows:
\begin{equation}
\label{length}
\lambda(M)=\lambda(M/IM)\big(\gamma_I(M)+1\big).
\end{equation}

Moreover, when $M$ $I$-free, we have 

\begin{equation}
\label{lengthfree}
\lambda(M)=b_0(M)\lambda(R/I)\big(\gamma_I(M)+1\big).
\end{equation}

\begin{remark}
Note that $\gamma_I(M)$ is also equal to $\displaystyle{\frac{\lambda
    (I M)}{\lambda(M/IM)}}$. The lower bound of $\gamma_I(M)$ is achieved
as follows: $\gamma(M)=0$ if and only if $I M=0$. If $M$ is $I$-free, then $\gamma_I(M)$ also has a natural upper bound, see Property \ref{property}(2).
\end{remark}

Next, we summarize some basic properties of the invariant $\gamma_I(M)$.

\begin{lemma}
\label{property}
Here we assume $(R,\fm)$ is a local Artinian ring, $I$ is a fixed ideal, and $M$ is a finite module.
\begin{enumerate}
    \item For any module $M$, we have $\lambda(M/IM) \leq b_0(M)\lambda(R/I)$ with equality if and only if $M$ is $I$-free.
    \item If $M$ is $I$-free , then $\gamma_I(M) \leq \gamma_I(R)$, with equality if and only if $M$ is free.
    \item If $I^2M = 0$, then $\gamma_I(M) \leq b_0(I)$.
    \item Let $N \subset M$ be a submodule, then $\gamma_I(M) = \gamma_I(M/N)$ if and only if $\gamma_I(M) = \dfrac{\lambda(N)}{\lambda(N\cap IM)}$.
    \item If $M$ and $N$ are $I$-free, then $\gamma_I(M\otimes N) \leq (\gamma_I(M)+1)\gamma_I(N)$.
\end{enumerate}
\end{lemma}

\begin{proof}
    (1) Consider the short exact sequence
        \[
        0 \rightarrow M_1 \rightarrow R^{b_0(\fm)} \rightarrow M \rightarrow 0
        \]
        where $M_1$ is the first syzygy. Then apply $-\otimes R/I$ and obtain $(R/I)^{b_0(M)} \twoheadrightarrow M/IM$. The surjection is an isomorphism if and only if $M$ is $I$-free.
        (2) Assume $M$ is $I$-free. Then, notice that 
        \[
        \gamma_I(M) = \frac{\lambda(M)}{\lambda(M/IM)} - 1
        \leq \frac{\lambda(R)b_0(M)}{\lambda(M/IM)} - 1 = \frac{\lambda(R)}{\lambda(R/I)} - 1 
        = \gamma_I(R)
        \] where the last equality follows from $\lambda(M/IM) 
        = b_0(M)\lambda(R/I)$.
        If moreover $\gamma_I(M) = \gamma_I(R)$, then 
        \[
        \lambda(M)\lambda(R/I) = \lambda(R)\lambda(M/IM) = \lambda(R) b_0(M)\lambda(R/I)
        \]
        i.e. $\lambda(M) = \lambda(R) b_0(M)$. This implies $M$ is free. \\
       (3) Assume $b_0(I) = n $ and $ I  = (a_1,..,a_n)$. Then there is a surjection $M^n \rightarrow \frac{IM}{I^2M}$ that sends $(x_i)$ to $\sum_ia_ix_i$. Observe that $(IM)^n$ lies in the kernel of this map and therefore it induces a surjection $(\frac{M}{IM})^n \rightarrow \frac{IM}{I^2M}$. This shows that if $I^2M = 0$, then we have $b_0(I)\lambda(M/IM) \geq \lambda(IM)$. In other words, $\gamma_I(M) \leq b_0(I)$.
        (4) We have by definition $\gamma_I(M) = \gamma_I(M/N)$ if and only if
        \[
        \frac{\lambda(M)}{ \lambda(M\otimes R/I)} = \frac{\lambda(M)- \lambda(N)}{ \lambda(M/N \otimes R/I)}
        \]
        equivalently 
        \[
        \lambda(M)\lambda(M/N \otimes R/I) = \lambda(M\otimes R/I)(\lambda(M) - \lambda(N))
        \]
        or 
        \[
            \lambda(M)[\lambda(M)-\lambda(N)- \lambda(\frac{IM + N}{N})] = [\lambda(M)-\lambda(IM)][\lambda(M)-\lambda(N)].
        \]
         In particular, if $M$ is $I$-free, then 
            \[
            \lambda(M)[\lambda(M)- \lambda(IM + N) )] = [\lambda(M)-\lambda(IM)][\lambda(M)-\lambda(N)]
            \]
            or
            \[
            \lambda(M)\lambda(IM+N) = \lambda(M)[\lambda(N)+\lambda(IM)] - \lambda(IM)\lambda(N).
            \]
            We can rewrite it as 
            \[
             \lambda(M) = \frac{\lambda(IM)\lambda(N)}{\lambda(N)+\lambda(IM)-\lambda(IM+N)}
            \]
            and notice that $\displaystyle \lambda(N)+\lambda(IM)-\lambda(IM+N) = \lambda(N \cap IM)$, since $\displaystyle \dfrac{IM+N}{N} \cong \dfrac{IM}{N\cap IM}$. 

        (5) Notice that there is always a surjection $M\otimes IN$ to $I(M\otimes N)$ and $\displaystyle \lambda(M\otimes IN) \leq \lambda(M)\lambda(IN)$. Hence 
        \[
        \gamma_I(M\otimes N) = \frac{\lambda(I(M\otimes N))}{\lambda(R/I \otimes M\otimes N)} \leq \frac{\lambda(M)\lambda(IN)}{\lambda((R/I)^{b_0(M)})\lambda((R/I)^{b_0(N)})} = (\gamma_I(M)+1)\gamma_I(N).
        \]
   
\end{proof}

Next, we give a lemma that generalizes the result in \cite[1.4]{huneke2014vanishing}. it turns out that their results can be generalized by looking at ideals $I$ and modules such that $M$ is $I$-free.

\begin{lemma}
\label{bettiandgamma}
Let $M$, $N$ be
$R$-modules such that $M$ is nonzero and $N$ is non-free. Assume for some $i > 0$, $N_{i-1},N_i$ are $I$-free.
    \begin{enumerate}
        \item If $\Tor_i(M,N) = 0$, then:
        \[
        \dfrac{b_i(N)}{b_{i-1}(N)} = \dfrac{\gamma_I(M)-\gamma_I(M \otimes N_{i-1})}{\gamma_I(M \otimes N_i) + 1} \leq \gamma_I(M).
        \]
        \item If in addition, $\fm IM = 0$, then:
        \[
        \frac{b_i(N)}{b_{i-1}(N)} = \gamma_I(M) - \gamma_I(M \otimes N_{i-1}).
        \]
        \item If $\fm IM = 0$ and $\Tor_{i-1}(M,N) = \Tor_{i}(M,N) = 0$, then:
        \[
        \frac{b_i(N)}{b_{i-1}(N)} = \gamma_I(M).
        \]
    \end{enumerate}
    \end{lemma}

\begin{proof}
    (1). Since $\Tor_i(M,N) = 0$, we have a short exact sequence
    \[
     0 \rightarrow M \otimes N_i \rightarrow M \otimes R^{b_{i-1}} \rightarrow M \otimes N_{i-1} \rightarrow 0
    \]
    and hence 
    \[
    \lambda(M\otimes R^{b_{i-1}}) - \lambda(M\otimes N_{i-1}) = \lambda(M\otimes N_i).
    \]
    We can rewrite the equality as 

    \begin{equation}
    \label{gammaIM}
    \begin{aligned}
    & (\gamma_I(M) + 1)\lambda(M/IM)b_{i-1}(N) \\
    -& (\gamma_I(M\otimes N_{i-1}) + 1)\lambda(M \otimes N_{i-1} \otimes R/I)\\
    =& (\gamma_I(M\otimes N_{i}) + 1)\lambda(M \otimes N_{i} \otimes R/I).
    \end{aligned}
   \end{equation}
    \\
    Notice that for $j = i-1, i$, since $N_j \otimes R/I$ is free over $R/I$, we have $\lambda(M \otimes N_j \otimes R/I) = \lambda(M/IM)b_j(N)$. Hence we can rewrite \ref{gammaIM} as 
    \[
     [\gamma_I(M)  - \gamma_I(M\otimes N_{i-1})]b_{i-1}(N)  = [\gamma_I(M\otimes N_{i}) + 1]b_i(N).
    \]
    (2) Notice that $M \otimes N_i$ lies in $\fm(M \otimes R^{b_{i-1}})$, and therefore by assumption, $I(M \otimes N_i) \subset \fm I(M \otimes R^{b_{i-1}}) = 0 $. As a result, $\gamma_I(M\otimes N_{i}) = 0$. \\
    (3) Since $\Tor_{i-1}(M,N) = 0$, we have $M \otimes N_{i-1}$ lies in $\fm(M \otimes R^{b_{i-2}})$ and $\gamma_I(M\otimes N_{i}) = 0$.
    \end{proof}

Next we give an example of non-free modules $M,N$ and an ideal $I$ satisfying the hypotheses of \ref{bettiandgamma}. It shows that there are non-trivial cases where $\gamma_I(M)$ gives a sharper upper bound for some Betti number of the module $N$ for $I$ different from $\fm$.

\begin{exam}
    Let $R = k[x,y]/(x^2,y^3,xy^2)$, $M = R/(y^2)$ and $N = R/(x)$. We have $\Tor_1(M,N) = 0$. We can also compute the first two terms of the resolution of $N$ as 
    \[
    R^2 \xrightarrow{(x,y^2)} R \xrightarrow{x} R \rightarrow N.
    \] 
     Now let $I =(x,y^2)$, and by the discussion in Notation \ref{IFree}, $N, N_1$ is free over $R/I$. We can also compute $\gamma_I(M) = \dfrac{\lambda(IM)}{\lambda(M/IM)} = \frac{2}{2} = 1$. Whereas $\gamma_{\fm}(M) = \dfrac{\fm M}{b_0(M)} =3$. Applying Lemma \ref{bettiandgamma}, we have 
     \[
      b_1(N) = 1 \leq \gamma_I(M) b_0(N) =b_0(N) = 1
     \]
     whereas 
     \[
     b_1(N) = 1 \leq \gamma(M) b_0(N) =3b_0(N) = 3.
     \]  
\end{exam}
 The conditions $\Tor_{i-1}(M,N)= \Tor_{i}(M,N)=0$ for some $i>0$, $N_{i-1}$ and $N_i$ are $I$-free, and $\fm IM = 0$ implies $\dfrac{b_i(N)}{b_{i-1}(N)} = \gamma_I(M)$. This suggests a converse question: if a certain ratio of consecutive Betti numbers of $N$ meets the $\gamma_I(M)$-bound, does that force $\Tor_i(M,N)=0$? The following proposition gives an affirmative answer. 

\begin{prop}
    \label{Prop:gammaandTor}
    Let $M$ be nonzero and $N$ be non-free. Assume for some $i>0$, $N_{i-1},N_i$ are $I$-free. Then we can conclude $\Tor_i(M,N) = 0$ under either of the following two conditions:
    \begin{enumerate}
        
        \item   $\dfrac{b_i(N)}{b_{i-1}(N)} \leq \dfrac{\gamma_I(M)-\gamma_I(M\otimes N_{i-1})}{\gamma_I(M\otimes N_i)+1}$.
        \item  The $R$-modules $M\otimes N_{i-1}$ and $M\otimes N_i$ are annihilated by the ideal $I$ and $\dfrac{b_i(N)}{b_{i-1}(N)} \leq \gamma_I(M)$.
    \end{enumerate}
     
\end{prop}

\begin{proof}
    (1)Apply $-\otimes M$ to the short exact sequence $0 \rightarrow N_{i} \rightarrow R^{b_{i-1}(N)} \rightarrow N_{i-1} \rightarrow 0$. We have 
    \[
    0 \rightarrow \Tor_{1}(M,N_{i-1}) \rightarrow M\otimes N_{i} \rightarrow M\otimes R^{b_{i-1}(N)} \rightarrow M\otimes N_{i-1} \rightarrow 0.
    \]

    Applying the length formula we have 
    \[
    \lambda(M \otimes N_{i}) = b_{i-1}(N)\lambda(M) - \lambda(M\otimes N_{i-1}) + \lambda(\Tor_{1}(M,N_{i-1})).
    \]
    Now we have 
    
    \begin{equation*}
    \begin{aligned}
    &(\gamma_I(M \otimes N_{i}) + 1 )\lambda(M \otimes N_{i}\otimes R/I)  \\
    = & b_{i-1}(N)(\gamma_I(M)+1)\lambda(M/IM)\\
    - & (\gamma_I(M\otimes N_{i-1})+1)\lambda(M\otimes N_{i-1}\otimes R/I) \\
     + & \lambda(\Tor_{1}(M,N_{i-1})).
    \end{aligned}
    \end{equation*}
    
    Since $N_{i-1}$ and $N_i$ are $I$-free, by the same argument as in the proof of \ref{bettiandgamma}, we obtain

     \begin{equation*}
    \begin{aligned}
    &(\gamma_I(M \otimes N_{i}) + 1 )b_{i}(N)\lambda(M/IM) \\
    =  &  b_{i-1}(N)(\gamma_I(M)+1)\lambda(M/IM) \\
    - & (\gamma_I(M\otimes N_{i-1})+1)b_{i-1}(N)\lambda(M/IM)  \\    
    + & \lambda(\Tor_{1}(M,N_{i-1})).
    \end{aligned}
    \end{equation*}

   Then since $M$ is nonzero, by Nakayama's Lemma we have $\lambda(M/IM)$ is nonzero and hence
     \begin{equation*}
    \begin{aligned}
    &(\gamma_I(M \otimes N_{i}) + 1 )b_{i}(N) \\
    = & (\gamma_I(M) - \gamma_I(M\otimes N_i))b_{i-1}(N)\\
    + & \dfrac{\lambda(\Tor_{1}(M,N_{i-1}))}{\lambda(M/IM)}.
    \end{aligned}
    \end{equation*}
    Dividing both sides of the equality by $\gamma_I(M\otimes N_i)+1)b_{i-1}(N)$, we have
     \begin{equation}
     \label{justonetime}
    \begin{aligned}
    \dfrac{b_i(N)}{b_{i-1}(N)} &= \dfrac{\gamma_I(M)-\gamma_I(M\otimes N_{i-1})}{\gamma_I(M\otimes N_i)+1} +\frac{\lambda(\Tor_1(M,N_{i-1})}{\lambda(M/IM)b_{i-1}(N)(\gamma_I(M\otimes N_i) +1)} \\
    &\geq \dfrac{\gamma_I(M)-\gamma_I(M\otimes N_{i-1})}{\gamma_I(M\otimes N_i)+1}.
      \end{aligned}
    \end{equation}
    Now if $ \dfrac{b_i(N)}{b_{i-1}(N)} \leq \dfrac{\gamma_I(M)-\gamma_I(M\otimes N_{i-1})}{\gamma_I(M\otimes N_i)+1}$, then we can conclude $\Tor_1(M,N_{i-1}) = \Tor_i(M,N) = 0$. \\
    (2) By assumption we have $\gamma_I(M\otimes N_{i-1}) = \gamma_I(M\otimes N_{i}) = 0$ and therefore, we have $ \dfrac{b_i(N)}{b_{i-1}(N)} \geq \gamma_I(M)$ from Equation \ref{justonetime}. Whence we can conclude $\Tor_i(M,N) = 0$.
\end{proof}

\begin{cor}
    Let $M$ be nonzero and $N$ be  non-free such that for some $i> 0$, $\Tor_i(M,N) = 0$ and $N_{i-1},N_i$ are $I$-free. If for some integer $j>0$, $M\otimes N_{j-1}$ and $M\otimes N_j$ are annihilated by the ideal $I$ and $N_{j-1},N_j$ are $I$-free, then $\Tor_j(M,N) = 0$.
\end{cor}

\begin{proof}
    By Lemma \ref{bettiandgamma}(a), \[ \dfrac{b_i(N)}{b_{i-1}(N)} = \dfrac{\gamma_I(M)-\gamma_I(M\otimes N_{i-1})}{\gamma_I(M\otimes N_i)+1} \leq \gamma_I(M).\] Now apply Proposition \ref{Prop:gammaandTor}(b), we have the desired result.
\end{proof}

Next we establish additional properties of $\gamma_I(M)$ assuming certain $\Tor$ vanishings.

\begin{lemma}
\label{one}
Let $M$, $N$ be $R$-modules such that $M$ is nonzero and $N$ is non-free.
\begin{enumerate}
\item If $\Tor_i(M,N) = 0$ and $N_i$ is $I$-free for $i \in [1,b_0(N)+1]$, then $\gamma_I(M) \geq 1$.
\item If $\fm IM = 0$, $\Tor_i(M,N) = 0$, and $N_i$ is $I$-free for $i \in [1,\log_2b_1(N)+1]$, then $\gamma_I(M)$ is an integer.
\end{enumerate}
\end{lemma}

The proof of the lemma involves the same idea as in \cite[1.6]{huneke2014vanishing}.
\begin{proof}
  (1) Assume that $\gamma(M)<1$. By Lemma \ref{bettiandgamma}(1) we have then
      \[
b_i(N)\leq\gamma(M) b_{i-1}(N)<b_{i-1}(N)\quad\text{for all}\quad
i\in[1,b_0(N)].
\]
in other words we have 
\[
b_i(N) \leq  b_{i-1}(N) -1
\]
for all $i \in[1,b_0(N)]$. By induction, we have
  $b_{b_0(N)}(N) \leq 0$ and we conclude that $b_{b_0(N)}(N)=0$, hence $N$ has
  finite projective dimension and it is thus free, contradicting the
  hypothesis.

 (2) Using Lemma \ref{bettiandgamma}(3) we have:
\[
b_{i+1}(N)=(\gamma(M))^ib_1(N)\quad\text{ for all}\quad i\in [1,\log _2
b_1(N)+1].
\]

Let $u,v$ be relatively prime positive integers such that
$\gamma(M)=uv^{-1}$. It follows that $v^i$ divides $b_1(N)$ for all
$i\in [1,\log _2 b_1(N)+1]$. If $v\ge 2$, then $b_1(N)\ge
2^i$ for all such $i$. In particular, $b_1(N) \ge 2^{\log _2 b_1(N)+1} = b_1(N) + 1$, a contradiction.
\end{proof}

Using the previous result, we can show that for two finite $I$-free modules $M$ and $N$ in certain rings, when $\Tor_i(M,N) = 0$ for $i$ large enough, then one of them has to be free.

Let $e$ denote $b_0(\fm)$, $h$ denote $b_0(I)$, and $t$ be the Loewy length of the ring $R$. 
\begin{prop}
        \label{prop2}
        Assume $(R,\fm)$ is a local ring with an ideal $I$ such that $\lambda(\fm^2) + 2 - b_0(I) < e-1$ and $b_0(I) \leq e$. Let $M,N$ be finite $R$-modules such that $\fm IM = 0$. If for all $i \gg 0$, $\Tor_i(M,N) = 0$ and $N_i$ are $I$-free, then $M$ or $N$ is free.
    \end{prop}

    \begin{proof}
        Assume by contradiction that neither $M$ nor $N$ is free. Then by \cite[2.2]{gasharov1990boundedness} we have for $i\gg 0$
        \[
        b_{i+2}(N) \geq eb_{i+1}(N) - (\lambda(\fm^2)+2 - t)b_i(N).
        \]
       
        Denote $a = \lambda(\fm^2)+2 - t$. By Lemma \ref{bettiandgamma} we can rewrite the above as  
        \[
        \gamma_I(M)^2b_i(N) - e\gamma_I(M)b_i(N) + ab_i(N) \geq 0.
        \]
        Since $N$ is non-free, we have 
        \[
        \gamma_I(M)^2 - e\gamma_I(M) + a \geq 0.
        \]
        Denote $r_1 = \frac{e+ \sqrt{e^2 - 4a} }{2}$ and $r_2 = \frac{e - \sqrt{e^2 - 4a} }{2}$. In particular by assumption $e^2 - 4a > e^2 - 4(e-1)$ $r_1$ and $r_2$ are real roots. Hence $\gamma_I(M)\leq r_2 < 1$ or $\gamma_I(M) \geq r_1 > e \geq b_0(I)$. Then $\gamma_I(M) \geq 1$ by \ref{one}(1), $\gamma_I(M)$ is an integer by \ref{one}(2), and $\gamma_I(M) \leq h$ by \ref{IMbetti}(1). Hence $\gamma_I(M) \leq r_2 < 1$ and $\gamma_I(M) = 0$. This is a contradiction.
    \end{proof}

\begin{lemma}
    \label{IMbetti}
   Let $M$ be a nonzero module and $N$ be a non-free module. Assume $\fm IM =0$, $N,N_1,N_2$ are $I$-free and $\Tor_{1}(M,N) = \Tor_{2}(M,N) = 0$. Then 
    \begin{enumerate}
        
    \item $
    b_1(M) = (b_0(I)-\gamma_I(M))b_0(M)$.

    \item If in addition, $M,M_1$ are $I$-free, then $\lambda(IM_1) = (\lambda(\fm I)+b_0(I) - \lambda(R/I)b_0(I))b_0(M)$.
    \end{enumerate}
    \end{lemma}

    \begin{proof} (1) Since $\Tor_{1}(M,N) = \Tor_{2}(M,N) = 0$, we have 
        \[
        b_1(M\otimes N_1) = b_0(M)b_2(N) + b_1(M)b_1(N).
        \]
        Since $\Tor_1(M,N) = 0$, we have 
        \[
        I(M\otimes N_1) \subset I(\fm(M\otimes R^{b_0(N)})) = 0
        \] and therefore $M\otimes N_1 \otimes R/I \cong M\otimes N_1$. Moreover, $M$ and $N_1$ are $I$-free, so $M\otimes N_1$ must be isomorphic to copies of $R/I$. Hence $b_1(M\otimes N_1) = b_0(I) \cdot b_0(M)b_1(N)$. While on the other hand, by Lemma \ref{bettiandgamma}, we have $b_1(N_1) = b_2(N) = \gamma_I(M)b_1(N).$ Together we have 

        \begin{equation}
            \begin{split}
            b_0(I)  b_0(M)b_1(N) &= b_0(M)b_2(N) + b_1(M)b_1(N)  \\
                                 &= b_0(M)b_1(N)\gamma_I(M) + b_1(M)b_1(N)
        \end{split}
        \end{equation}
        and since $N$ is non-free,  
        \[
        b_1(M) = (b_0(I) -  \gamma_I(M) )b_0(M).
        \]
        (2) By the definition of $\gamma_I(M)$ and the virtue that $M/IM$ is free over $R/I$, we have 
        \[
        \lambda(M) = (\gamma_I(M) + 1)\lambda(M/IM) = (\gamma_I(M) + 1)\lambda(R/I)b_0(M).
        \]
        While looking at the short exact sequence $0 \rightarrow M_1 \rightarrow R^{b_0(M)} \rightarrow M \rightarrow 0$, we have
        \[
        \lambda(M_1) = \lambda(R)b_0(M) - \lambda(M).
        \]
        Now plugging the $\lambda(M)$ calculated above 
        \begin{equation}
        \begin{split}
        \lambda(M_1)& = (\lambda(\fm I) + b_0(I) + \lambda(R/I) - \gamma_I(M)\lambda(R/I) - \lambda(R/I))b_0(M)  \\
        &=(\lambda(\fm I) + b_0(I) - \gamma_I(M)\lambda(R/I)  )b_0(M). 
        \end{split}
        \end{equation}
        Then we combine what we calculate for $\lambda(M_1)$ and the assumption $M_1$ is $I$-free and by (a)
        \begin{equation}
        \begin{split}          
        \lambda(IM_1) = &\lambda(M_1) - \lambda(M_1/IM_1) \\
        =&  \lambda(M_1) - b_1(M)\lambda(R/I) \\
        =& (\lambda(\fm I) + b_0(I) - \gamma_I(M)\lambda(R/I)  )b_0(M) -(b_0(I) -  \gamma_I(M))b_0(M)\lambda(R/I)  \\
        =& (\lambda(\fm I)+b_0(I) - \lambda(R/I)b_0(I))b_0(M).
        \end{split}
        \end{equation}
        
    \end{proof}

\begin{lemma}
\label{sum}
Let $M$, $N$ be non-free $R$-modules with $\fm IM=\fm IN=0$. If $\Tor^R_2(M,N)=\Tor^R_1(M,N)=0$ and $M,M_1,M_2,N, N_1$ are $I$-free, then
$ \displaystyle \gamma_I(M)+\gamma_I(N)-\gamma_I({M\otimes_RN})=b_0(I)$.
 \end{lemma}

\begin{proof}  Compare by Lemma \ref{IMbetti}(1)
\[
\frac{b_1(M)}{b_{0}(M)} = (b_0(I)-\gamma_I(M))
\]
and by Lemma \ref{bettiandgamma}(2) with $M$ and $N$ switched
\[
\frac{b_1(M)}{b_{0}(M)} = \gamma_I(N) - \gamma_I(M \otimes N).
\]
\end{proof}

\section{Vanishing of Tor and Freeness of A Module}
\label{Vanishing of Tor and Freeness of A Module}
The behavior of Betti numbers of finite $R$-modules for the local ring $(R,\fm)$ such that $\fm^3 = 0$ was
studied by Lescot \cite{LESCOT1985287}. We adapt some techniques from Lescot to a more general situation that there is a fixed ideal $I$ such that $\fm^2I = 0$.

        \begin{lemma}
        \label{purebetti}
        Let $M,N$ be finite modules such that $M$ is nonzero and $N$ is non-free. Assume $\fm IM = 0$. Let $i \geq 0$ be an integer such that $  M_{i},M_{i+1},N_{i},N_{i+1},N_{i+2}$ are $I$-free and $\Tor_{i+1}(M,N) = \Tor_{i+2}(M,N) =  0$, then 
        \begin{equation*}       
        \lambda(M_{i+1}/IM_{i+1}) =  b_0(I)\lambda(R/I)b_i(M) - \lambda(IM_i).       
        \end{equation*}
      
      \end{lemma}

    \begin{proof}
    Breaking down the length of modules via short exact sequences, we have 
       \begin{equation*}
        \begin{aligned}
       \lambda(M_{i+1}/IM_{i+1}) &= \lambda(M_{i+1})-\lambda(IM_{i+1})\\
        &=\lambda(R)b_i(M)  - \lambda(M_i) - \lambda(IM_{i+1})\\  
        &=\big(\lambda(\fm I) + b_0(I) + \lambda(R/I)\big)b_i(M)- \lambda(M_i/IM_i) -\lambda(IM_{i+1})- \lambda(IM_i).
        \end{aligned}
        \end{equation*}
      Now since $M_i$ is $I$-free and apply Lemma \ref{IMbetti}(b) to $\lambda(IM_{i+1})$ (here is why we need the syzygies of $M$ and $N$ to be $I$-free), we have the above 
        \begin{equation*}
        \begin{aligned}
        = &(\lambda(\fm I) + b_0(I) + \lambda(R/I))b_i(M) \\
        - &\lambda(R/I)b_i(M) -(\lambda(\fm I)+b_0(I)-\lambda(R/I)b_0(I))b_i(M)- \lambda(IM_i) \\
        = & b_0(I)\lambda(R/I)b_i(M) - \lambda(IM_i).
        \end{aligned}
        \end{equation*}  
    \end{proof}

    \begin{chunk}
    In \cite{huneke2014vanishing}, one crucial fact from Lescot is that $k$ is not a direct summand of a module $N$ if and only if $\fm N \subset \Soc(N)$. Moreover the fact that $\fm^2 \ne 0$, $N$ is non-free, and $\Tor_i(M,N) = 0$ for some non-free module $M$ implies that $k$ is not a direct summand of $N_j$ for $j\leq i$. Hence the length of the module $\fm N$ can be computed via the module $\Soc(N)$ in such situations. If in addition $\Soc(R) = \fm^2$, then $b_0(\fm N_i) = \dim \Soc(R) b_{i-1}(N)$. Here we believe that $\Soc(N) \subset \fm N$ is the assumption that plays a critical role. We obtain 
    \end{chunk}

     \begin{lemma}
        \label{socle}
        Assume $\Soc(R) = \fm I$. If $M$ is a non-free syzygy module and $\Soc(M_1) \subset IM_1$, then $\Soc(M_1) = IM_1 = mIR^{b_0(M)}$.
    \end{lemma}

    \begin{proof}
        Notice that $M$ being a syzygy module of $R$ implies $\fm IM = 0$.
        Looking at the short exact sequence 
        \[
        0 \rightarrow M_1 \rightarrow R^{b_0(M)} \rightarrow M \rightarrow 0,
        \]
        we have $M_1  \subset mR^{b_0(M)}$ and therefore $\fm IM_1 = 0$. Hence, we have $\Soc(M_1) = IM_1$. However, $\fm IM = 0$ implies $\fm IR^{b_0(M)} \subset M_1$. Hence
        \[
        \fm IR^{b_0(M)} = \Soc(\fm IR^{b_0(M)}) \subset \Soc(M_1) \subset \Soc(\fm R^{b_0(M)}) = \fm IR^{b_0(M)}.
        \]
    \end{proof}

    However, comparing \ref{IMbetti}(2) and \ref{socle}, it can be deduced that 
    \begin{cor}
        Let $M$ be a non-zero module and $N$ be a non-free module. Assume $\fm IM =0$ and $M,M_1,N,N_1,N_2$ are $I$-free and $\Tor_{1}(M,N) = \Tor_{2}(M,N) = 0$. If $\Soc(R) = \fm I$ and $\Soc(M_1) = IM_1$, then $\lambda(R/I) = 1$ and $I = \fm$.
    \end{cor}
    In other words, assuming $\Soc(R) = \fm I$ and $\Soc (N_i) = IN_i$ for some $i$ essentially forces $I = \fm$ and the results about rings with $\fm^3 = 0$ have been proved in \cite{huneke2014vanishing}. Hence, some of the results do not seem to generalize.\\ 
    However, we can obtain some results assuming only one of the socle conditions.

     \begin{lemma}
        \label{betti1socle}
        Let $M$ be a non-free module. Assume $\Soc(R) = \fm I$.
        Then 
        \[
       \lambda(M_1/IM_1) \geq b_0 (M) \cdot b_0(I) - \lambda(IM).
        \]

    \end{lemma}

    \begin{proof}
        We have $\lambda(M_{1}/IM_{1}) = \lambda(M_{1}) - \lambda(IM_{1}) $. For simplicity, we denote $b_i = b_i(M)$ as the $i$-th Betti number of $M$. But since $\fm IM_{1} = 0$,    $\lambda(IM_{1}) \leq \lambda(\Soc(M_{1})) =  b_0 \cdot \lambda(\fm I)$. The equality follows from the last line in the proof of \ref{socle}.\\
        On the other hand, notice that $\lambda(M_{1}) = b_0 \cdot \lambda(R) - \lambda(M)$. Hence we can rewrite 

        \begin{equation}
        \label{mess1}
        \lambda(M_1/IM_1)  \geq b_0 \cdot \lambda(R) - \lambda(M) - b_0\cdot\lambda(\fm I)
        \end{equation}
        and similarly we have 
        \[
        \lambda(M) = \lambda(M/IM) + \lambda(IM) \leq b_0 \cdot\lambda(R/I) + \lambda(IM).
        \]
        Then, notice that 
        \[
        \lambda(R) = \lambda(R/I) + b_0(I) + \lambda(\fm I).
        \]
        Now substitute them into \ref{mess1} we have 
        \begin{equation}
            \begin{aligned}
       \lambda(M_1/IM_1) &\geq b_0 \cdot (\lambda(R/I) + b_0(I) + \lambda(\fm I)) - b_0 \cdot\lambda(R/I) - \lambda(IM) - b_0\cdot\lambda(\fm I)\\
        &= b_0  \cdot b_0(I) - \lambda(IM).
         \end{aligned}
        \end{equation} 
        
    \end{proof}

    Before starting the first main theorem, we have some notations for simplicity

    \begin{nota}
        Denote $s =\lambda(R/I) $, $h = b_0(I)$, and $c = \lambda(\fm I)$. Note that $\displaystyle s+h+c = \lambda(R)$.
    \end{nota}

        \begin{theorem}
        \label{ThreeVanishP}
        Let $(R,\fm)$ be a local Artinian ring with a fixed ideal $I$ such that $\fm^2I = 0$. Let $M,N$ be nonzero, non-free modules such that $\fm IM = 0$ and $\fm IN = 0$. If there exists $j > 0$  such that for $t = j, j+1, j+2$
        \begin{enumerate}
            \item $\Tor_t(M,N) = 0 $ and
            \item $M_t,N_t$ are $I$-free,
        \end{enumerate}
        
        then $\gamma_I(M)$ and $\gamma_I(N)$ satisfy the equation $ \displaystyle{ s\gamma^2 - sh\gamma + c+h-sh = 0 }$.
        \end{theorem}

        \begin{proof}
        By Lemma \ref{bettiandgamma} (3), we have 
        \begin{equation}
        \label{eq:ThreeVanishP1}
        b_{j+2}(N) = \gamma_I(M)b_{j+1}(N) = \gamma_I(M)^2b_j(N) \neq 0.
        \end{equation}
        
        By Lemma \ref{purebetti} we also have
        \[
        sb_{j+2}(N) =shb_{j+1}(N) - (c+h-sh)b_{j}(N).
        \]
        Now replacing $b_{j+2}(N)$ and $b_{j+1}(N)$ using equation \ref{eq:ThreeVanishP1}, we have
        \[
        s\gamma_I(M)^2b_{j}(N) =hs\gamma_I(M)b_j(N) - (c+h-sh)b_{j} (N).
        \]
        Since $N$ is non-free, \[          s\gamma_I(M)^2 -hs\gamma_I(M) + (c+h-sh) = 0.        \] 
        \end{proof}

     \begin{cor}
        Let $(R,\fm)$ be an Artinian ring and $I$ be a nonzero ideal such that $\fm^2 I = 0$ and $ \lambda(R) < 2sh $. Assume that $M$ is a nonzero module such that $\fm IM = 0$ and $N$ is non-free. Let $j > 1$ be an integer such that for $t = j, j+1, j+2$,  $M_t,N_t$ are $I$-free. Set $\beta = \max \{\log_2(b_1(M))+1,\log_2(b_1(N))+1, j+ 2 \}$. If $\Tor_i(M,N) = 0 $ for $i \in [1,\beta]$, then $M$ or $N$ is free.
        \end{cor}

    \begin{proof}
        Assume by contradiction that $M$ is non-free. Then we can apply Theorem \ref{ThreeVanishP}, so $\gamma_I(M)$ satisfies the equation 
        \[
        \gamma^2 - h\gamma + \frac{c+h-sh}{s} = 0.
        \]
        For simplicity, denote $\alpha = \frac{c+h-sh}{s}$.
       Hence $\gamma_I(M)$ lies in the set $\{\frac{h + \sqrt{h^2 - 4\alpha}}{2}, \frac{h - \sqrt{h^2 - 4\alpha}}{2}\}$. Notice that by the assumptions $4\alpha < 4h- 4$ and $h -1 \geq 0$, hence $\sqrt{h^2 -4\alpha} > h-2 $. As a result, if we denote $\gamma_1 \leq \gamma_2$ as the roots of $\gamma^2 - h\gamma + \frac{c+h-sh}{s} = 0$, $\gamma_1 = \frac{h - \sqrt{h^2 - 4\alpha}}{2} < \frac{h-h+2}{2} = 1 $ and $\gamma_2  = \frac{h + \sqrt{h^2 - 4\alpha}}{2} > \frac{h + h -2}{2} = h - 1$. By Lemma \ref{one}(b), we have that $\gamma_I(M)$ is an integer. Since $M$ is non-free, so $b_1(M) \ne 0$ and therefore by Lemma \ref{IMbetti}(a), we have $\gamma_I(M) < h$. Therefore, we must have $\gamma_I(M) = \gamma_1 = 0$. However, since $\Tor_j(M,N_1) = \Tor_{j+1}(M,N_1) = 0$ and $\fm IN_1 = 0$, we have that by Lemma \ref{bettiandgamma}(c) $b_1(N_1) = \gamma_I(M)b_0(N_1) = 0$, which contradicts the assumption that $N$ is non-free. Hence we conclude $M$ must be free.\end{proof}

    \begin{chunk}
    \label{omega}
    Next we collect some results from \cite[2.7]{huneke2014vanishing}. 
            For any Artinian ring $R$ (not necessarily with $\fm^2I = 0$) with a dualizing module $\omega$, we have $b_0(\omega)$ = $\dim_k \Soc(R)$, $\dim_k \Soc(\omega) = 1$, and $\lambda(R) = \lambda(\omega)$.
     \end{chunk}

     \begin{remark}
    Assume $\Soc(M) = IM$ and $\Soc(M^{\dagger}) = IM^{\dagger}$  where $(-)^{\dagger} = \Hom(-,\omega)$. Then by Matlis duality we have 
        \[
        \lambda(IM^{\dagger}) = \dim_k \Soc(M^{\dagger}) = b_0(M)
        \]
        and 
        \[
        \Soc(M^{\dagger}) = \Hom(k,\Hom(M,\omega)) \cong \Hom(k,\omega)^{b_0(M)}
        \]
        so 
        \[
        \lambda(IM) = \lambda(IM^{\dagger\dagger}) = \lambda(  \Soc(M^{\dagger\dagger})) = b_0(M^{\dagger}).
        \]
        Hence 
        \begin{equation}
            \begin{aligned}        
        \gamma_I(M) &= \frac{\lambda(IM)}{\lambda(M/IM)} &\geq& \frac{b_0(M^{\dagger})}{b_0(M)\lambda(R/I)} &=\frac{b_0(M^{\dagger})}{\lambda(IM^{\dagger})\lambda(R/I)} \\
        &= \frac{1}{\lambda(R/I)^2} \frac{b_0(M^{\dagger})\lambda(R/I)}{\lambda(IM^{\dagger})}  &\geq&    \frac{1}{\lambda(R/I)^2} \frac{\lambda(M^{\dagger}/IM^{\dagger})}{\lambda(IM^{\dagger})} &= \frac{1}{\lambda(R/I)^2} \frac{1}{\gamma_I(M^{\dagger})}. 
        \end{aligned}
        \end{equation}
        The first inequality is equality if $M$ is $I$-free and the second is equality if $M^{\dagger}$ is $I$-free.
        \end{remark}

        \begin{prop}      
        Let $(R,\fm)$ be an Artinian local ring with a canonical module $\omega$ and $I$ be an ideal different from $\fm$. Assume that there exists a finite module $M$, such that $\Soc(M) = IM$, $\Soc(M^{\dagger}) = IM^{\dagger}$, and $\fm IM = \fm IM^{\dagger}=0$. If for all $i \gg 0$, both $M_i$ and $M^{\dagger}_i$ are $I$-free and $\Tor_i(M,M^{\dagger}) = 0$, then $M$ or $M^{\dagger}$ is free.
        \end{prop}

        \begin{proof}
        Assume by contradiction that neither of them is free. By Lemma \ref{bettiandgamma}, there exists an integer $n$ such that for $i>n$
        \[
        \gamma_I(M^{\dagger})b_i(M) = b_{i+1}(M)
        \]
        and 
        \[
        \gamma_I(M)b_i(M^{\dagger}) = b_{i+1}(M^{\dagger}).
        \]
        Also by the remark we have 
        \[
        \gamma_I(M)\gamma_I(M^{\dagger}) = \frac{1}{\lambda(R/I)^2} < 1.
        \]
        Hence $ \gamma_I(M)$ or $ \gamma_I(M^{\dagger})$ is less than $1$ and therefore $M$ or $M^{\dagger}$ is free.
    \end{proof}

    \section{Betti Numbers of the Canonical Module}
    \label{Betti Numbers of the Canonical Module}
    In this section, we give some partial answer to the question posed in \cite{jorgensen2007growth}, "Does $b_1(\omega) \leq b_0(\omega)$ imply $R$ is Gorenstein?"
    We still assume all the rings to be local Artinian, with a fixed ideal $I$.
    First, we can obtain an upper bound for the Betti growth of $\omega$ while certain $Tor$ modules of the canonical module vanish. Before we start, let's recall some notations that will be frequently used in this section. We denote $s =\lambda(R/I) $, $h = b_0(I)$, and $c = \lambda(R) - s -h = \lambda(\fm I)$.
    \begin{prop}
        \label{bettiomega}
         If $\Tor_1(\omega, \omega) = \Tor_2(\omega, \omega) = 0$ and $\omega,\omega_1,\omega_2$ are $I$-free, then
         \[
          \frac{b_1(\omega)}{b_0(\omega)} \leq \frac{b_0(I)}{2}.
         \]
        
    \end{prop}

    \begin{proof}
         Assume $\omega$ is non-free. Set $b_i = b_i(\omega)$. By Lemma \ref{bettiandgamma}, we have 
        \[
           \frac{b_1}{b_0}  \leq \gamma_I(\omega) .
        \]
        While on the other hand, by Lemma \ref{IMbetti}, we have 
        \[
        \frac{b_1}{b_0} = h - \gamma_I(\omega).
        \]
        Hence we can conclude $\gamma_I(\omega) \geq \dfrac{h}{2}$ and the result follows. 
    \end{proof}

    \begin{remark}
    Now we give some calculations on $\omega$ over the ring $(R,\fm)$ with an ideal $I$ such that $\fm^2I = 0$. If $\fm I \ne 0$, then since $\fm I\omega \subset \Soc(\omega)$ and $\dim_k \Soc(\omega) = 1$, we have $\lambda(\fm I\omega) \leq b_0(\fm I\omega) = 1$. But moreover $\fm I \neq 0$, so $\lambda(\fm I\omega) = 1$. Now denote $b_0(\omega) = \dim_k\Soc(R) = a$ and set $N = \omega_1$. If $\omega$ is $I$-free, then we have 
    \begin{enumerate}
            \item $\lambda(\omega) = \lambda(R) = \lambda(R/I) + b_0(I) + \lambda(\fm I) = s + h + c$.
            \item $\lambda(I\omega) = \lambda(\omega) - \lambda(\omega/I\omega) = s + h + c - as$.
    \end{enumerate}
\end{remark}
    \begin{prop}
        \label{prop3.1}
         Let $\Soc(R) = \fm I$. Assume $h > 1$. If $\lambda(\omega_1/I\omega_1) \leq \lambda(\omega/I\omega)$, then 
         \[
        \lambda(\fm I) \leq \frac{\lambda(R/I) + b_0(I)}{b_0(I)-1}.
         \]
         In particular if $h > s+2$, then $R$ is Gorenstein.
    \end{prop}
    \begin{proof}
        Set $N = \omega_1$. Recall that $c = \lambda(\fm I)$ Apply \ref{betti1socle}, we have 
        \[
        \lambda(N/IN) \geq ch - s-h-c + cs.
        \]
        while on the other hand
        \[
        \lambda(\omega/I\omega) \leq b_0(\omega)\lambda(R/I) = cs
        \]
        hence we get
        \[
         cs \geq ch - s-h-c + cs.
        \]
        Simplifying it, we have 
        \[
        c \leq \frac{s+h}{h-1}.
        \]
        If $h > s+2 $, then $\frac{s+h}{h-1} < 2 $ and $c = 1$, i.e. $\dim_k \Soc(R) = 1$. $R$ is Gorenstein.
    \end{proof}

    \begin{cor}
        \label{cor3.4}
        Let $(R,\fm)$ be an Artinian ring such that $\Soc(R) = \fm^2$ and $b_0(\fm) \geq 4$. If $b_1(\omega) \leq b_0(\omega)$, then $R$ must be Gorenstein.
    \end{cor}

    \begin{proof}    
        Applying Proposition~\ref{prop3.1} with $I=\fm$ and $s=1$, the result follows.

    \end{proof}

    \begin{cor}
     \label{cor3.5}
        Let $(R,\fm)$ be a local Artinian ring such that $\fm^2I = 0$. Assuming (i) $h > 2$; (ii) $b_2(\omega) \leq b_1(\omega)$. If there exists a finite non-free module $M$ such that for three consecutive integer values $i>1$ $\Tor_{i}(M,\omega) = 0$, $\fm IM = 0$ and that $M_t,\omega_1,\omega_2$, and $\omega_3$ are $I$-free for $i \leq t \leq i+3$. Then $R$ is Gorenstein.
\end{cor}

\begin{proof} 
         Assume by contradiction $\omega$ is non-free. Then we can replace $M$ by $M_{i-1}$ to get $\Tor_1(M,N) = \Tor_2(M,N) = \Tor_3(M,N) = 0$. Let $N = \omega_1$. Then we have $\fm IN = 0$. Now apply \ref{sum}, we have $\gamma_I(M) + \gamma_I(N) = h + \gamma_I(M\otimes N)$. In particular, at least one of $\gamma_I(M)$ and $\gamma_I(N)$ is no less than $\frac{h}{2}$. If $\gamma_I(M) \geq \frac{h}{2} > 1$, by Lemma \ref{bettiandgamma}, we also have $b_1(N) = \gamma_I(M)b_0(N)$ and therefore it contradicts $b_2(\omega) \leq b_1(\omega)$. If $\gamma_I(N) > \frac{h}{2}$, we have by Lemma \ref{IMbetti}(1), $b_1(N) = (h -\gamma_I(N))b_0(N) > \frac{h}{2} b_0(N)$ which is also a contradiction. 
    \end{proof}

\end{document}